\documentclass[a4paper, 11pt, reqno]{amsart}

\usepackage{amssymb}
\usepackage{amsthm}
\usepackage{bm}
\usepackage{latexsym}
\usepackage{amsmath}
\usepackage{mathrsfs}
\usepackage{caption}
\usepackage{amscd}
\usepackage[all,cmtip]{xy}
\usepackage[usenames]{color}
\usepackage{graphicx}
\usepackage{color}
\usepackage{amscd}
\usepackage{float}
\usepackage{graphics}
\usepackage{tikz}
\usepackage{tikz-cd}
\usepackage{comment}
\usepackage{xspace}
\usepackage{mathtools}
\usepackage{subfig}
\usepackage{enumerate}
\usepackage{hyperref}

\usepackage{marginnote}

\newcommand{\appsection}[1]{\let\oldthesection\thesection
\renewcommand{\thesection}{Appendix \oldthesection}
\section{#1}\let\thesection\oldthesection}

\DeclareMathOperator{\coeff}{coeff}

  \newtheorem{introthm}{Theorem}

  \newtheorem{introcor}{Corollary}

  \newtheorem{theorem}{Theorem}[section]
  \newtheorem{lemma}[theorem]{Lemma}
  \newtheorem{proposition}[theorem]{Proposition}
  \newtheorem{corollary}[theorem]{Corollary}

  \newtheorem{definition}[theorem]{Definition}
  \newtheorem{example}[theorem]{Example}

\newtheorem{remark}[theorem]{Remark}

\theoremstyle{remark}

\numberwithin{equation}{section}

\calclayout

\title{Fundamental Groups and Coregularity:\\Klt Pairs with Trivial and Nef Log-Anticanonical Class}

\author[L.~Braun]{Lukas Braun}
\address{Institut für Mathematik, Fakultät für Mathematik, Informatik und Physik, Universität Innsbruck, Technikerstraße 13, A-6020 Innsbruck, Austria}
\email{lukas.braun@uibk.ac.at}

\thanks{
During the work on this manuscript, LB was supported by FWF grant 10.55776/PAT9495823. 
}

\author[F.~Figueroa]{Fernando Figueroa}
\address{Department of Mathematics, Northwestern University, Evanston, IL 60208, USA
}
\email{fernando.figueroa@northwestern.edu}

\subjclass[2020]{Primary 14E30, 14F35, Secondary 14J32.}

\newcommand{\reg}{\mathrm{reg}}
\newcommand{\st}{\mathrm{st}}
\newcommand{\orb}{\mathrm{orb}}
\newcommand{\qtilde}{\widetilde q}
\newcommand{\bM}{\mathbf M}
\newcommand{\bN}{\mathbf N}
\newcommand{\Gam}{\pi_1^{\reg}}
\newcommand{\Supp}{\operatorname{Supp}}

\newcommand{\rank}{\operatorname{rank}}
\newcommand{\num}{\equiv}
\date{October 14, 2026} 
\hypersetup{hidelinks}
\begin{document}

\begin{abstract}
In this article, we study how the absolute coregularity of a projective log pair
reflects on its fundamental group. More precisely, we show that for a projective
klt log pair $(X,D)$ of absolute coregularity $c$ and arbitrary dimension
the fundamental group $\pi_1^{\rm reg}(X,D)$ admits a normal abelian subgroup
of finite index and rank at most $2c$.
We reduce this statement to klt Calabi-Yau pairs with standard coefficients
of dimension $c$. In the cases $0\le c\le2$ and fixed dimension, we can
furthermore bound the index of a solvable normal subgroup.
Combining the reduction with the recent regular-locus theorem of
Fu-Guo-Song-Wang, we prove the coregularity bound in every dimension.
In particular, for a projective klt pair $(X,D)$ with $-(K_X+D)$ nef,
the orbifold fundamental group of the smooth locus is virtually abelian of rank at most
$2\dim X$.
\end{abstract}

\maketitle
\setcounter{tocdepth}{1}
\tableofcontents
\section{Introduction}
In this article, we investigate the connections between two important invariants of projective varieties (or pairs): the \textit{fundamental group} and the \textit{coregularity}.

The \textit{fundamental group} as a topological invariant gains its importance from viewing smooth projective varieties as Kähler manifolds. As such, it is long known that Fano manifolds are simply connected and projective manifolds with trivial or antinef canonical divisor have almost abelian fundamental group; see~\cite{Pau97} for the latter case. During the last few years, many of these statements have been generalized to orbifolds and to varieties with Kawamata log terminal (klt) singularities, with respect to the \textit{orbifold fundamental group} and/or the \textit{fundamental group of the smooth locus}. While proofs in the orbifold case mostly mimic the differential-geometric proofs from the smooth case using orbifold charts, in the klt case, often new techniques and insights are needed. In particular, the transfer from the fundamental group of the whole space to the one of the smooth locus often requires the local finiteness result from~\cite{Bra21}, which simultaneously proved the finiteness of the orbifold fundamental group of the smooth locus of klt log Fano pairs.

In the present article, we combine a reduction to standard-coefficient Calabi-Yau pairs with
the regular-locus theorem of Fu-Guo-Song-Wang~\cite{FGSW26}.
This gives virtual abelianity of the fundamental group for projective klt pairs with a
generalized lc Calabi-Yau completion, including all projective
klt pairs with nef log anticanonical divisor.

We write
\[
 \Gam(X,D)=\pi_1^{\orb}(X_{\reg},D_{\st}|_{X_{\reg}}),
\]
where $D_{\st}$ is the standard approximation of $D$, and call this group the regional fundamental group of the pair $(X,D)$.
Throughout, boundaries and nef b-divisors are rational, and varieties are connected and normal over $\mathbb C$.
A klt pair $(X,D)$ requires $K_X+D$ to be $\mathbb Q$-Cartier.
For a finitely generated virtually abelian group, its \emph{virtual rank}
is the integer $r$ for which it has a finite-index subgroup isomorphic
to $\mathbb Z^r$. This integer is independent of the subgroup chosen.

\subsection{Coregularity and fundamental groups}
The coregularity measures the difference between the dimension of
a variety and the dimension of the dual complex of a generalized lc
Calabi-Yau structure; see~\cite{FMP22,Mor22}. In this paper, absolute coregularity minimizes
this invariant over generalized lc Calabi-Yau completions with
larger boundary and nef part.
Recall that a completion consists of $B\ge D$ and a b-nef,
b-$\mathbb Q$-Cartier divisor $\bM$ such that $(X,B,\bM)$ is
generalized lc and $K_X+B+\bM_X\sim_{\mathbb Q}0$.

A klt Calabi-Yau pair means a klt pair
with $K_X+D\sim_{\mathbb Q}0$. For such pairs, the absolute coregularity equals
their dimension. With the nef part held fixed, this follows from
\cite[Remark~2.26]{FMP22}; Lemma~\ref{lem:cy-coreg} proves the assertion
for the convention used here, which also allows the nef part to vary.
Our central reduction to pairs of dimension equal to the coregularity is the following:

\begin{introthm}\label{introthm-induction-coreg}
Fix an integer $c\ge0$. Assume that for every projective klt Calabi-Yau pair
of dimension $c$ with standard coefficients the regional orbifold
fundamental group is virtually abelian of rank at most $2c$.
Then every projective klt pair $(X,D)$ of absolute coregularity
$a\le c$ has virtually abelian regional orbifold fundamental group
of rank at most $2a$.
\end{introthm}

In recent work, Fu-Guo-Song-Wang prove that for a projective klt variety $V$
with $K_V\num0$, the group $\pi_1(V_{\reg})$ is virtually abelian
of rank $2\qtilde(V)$~\cite[Theorem~1.3]{FGSW26}.
Here the augmented irregularity $\qtilde(V)$ is the supremum of
$q(V')=h^1(V',\mathcal O_{V'})$ over connected finite quasi-\'etale
covers $V'\to V$.
The log index-one cover from~\cite[Proposition~12]{CGG24} proves the
standard-coefficient Calabi-Yau hypothesis of Theorem~\ref{introthm-induction-coreg}
in every dimension; see Proposition~\ref{prop:cy-cover}.
We therefore obtain the following all-dimensional conclusion.
\begin{introthm}\label{intro:coreg-all}\label{intro:completion} 
Let $(X,D)$ be a projective klt pair admitting a generalized lc
Calabi-Yau completion. Put
$c=\widehat{\operatorname{coreg}}(X,D)$.
Then $\Gam(X,D)$ has a normal free abelian subgroup of finite
index and rank at most $2c$.
\end{introthm}

In particular, pairs with absolute coregularity zero have a finite regional
fundamental group.

\subsection{Calabi-Yau pairs and nef log anticanonical divisors}
As a direct consequence for Calabi-Yau pairs, we obtain the following:
\begin{introthm}\label{intro:cy-all}\label{intro:cy-standard}
Let $(X,D)$ be a projective klt pair of dimension $n$ with
$K_X+D\num0$.  Let $s$ be the smallest positive integer such that
$s(K_X+D)\sim0$. Then $\Gam(X,D)$ is virtually abelian of rank
at most $2n$. In particular:
\begin{enumerate}
	\item If $D$ has standard coefficients, there is a connected
	finite cyclic log index-one cover $p:Y\to X$ of degree $s$, with $Y$ canonical
	and $K_Y\sim0$, and an exact sequence
	\[
	1\longrightarrow\pi_1(Y_{\reg})\longrightarrow\Gam(X,D)
	\longrightarrow\mathbb Z/s\mathbb Z\longrightarrow1.
	\]
	In this case the virtual rank is exactly $2\qtilde(Y)$. 
	\item If $D=0$, the virtual rank is exactly $2\qtilde(X)$.
\end{enumerate}

\end{introthm}

On the other hand, we arrive at the following consequence for pairs with nef log-anticanonical class:
\begin{introthm}\label{intro:nef}
Let $(X,D)$ be a projective klt pair of dimension $n$ such that
$-(K_X+D)$ is nef. Then $\Gam(X,D)$ is virtually abelian of
rank at most $2\widehat{\operatorname{coreg}}(X,D) \leq 2n$.
\end{introthm}

These results confirm \cite[Conjecture~1.1]{FGSW26nef} for projective
varieties $X$ with klt singularities and rational boundary $\Delta$,
even when $(X,\Delta)$ is only log canonical. Indeed, if
$-(K_X+\Delta)$ is nef, then
$(X,\Delta,\overline{-(K_X+\Delta)})$ is a generalized lc
Calabi-Yau completion of $(X,0)$, so
Theorem~\ref{intro:coreg-all} applies to
$\Gam(X,0)=\pi_1(X_{\reg})$.
When the pair is klt, Theorem~\ref{intro:nef} gives the stronger
orbifold conclusion. We note that virtual abelianity can fail
for orbifold fundamental groups of lc Calabi-Yau pairs;
see~\cite[Example~10.2]{GLM23}.

For $D=0$, the surjection
$\pi_1(X_{\reg})\twoheadrightarrow\pi_1(X)$ shows that
Theorem~\ref{intro:nef} recovers the projective boundary-free cases
of~\cite[Corollary~1.1]{FGSW26nef}, which concerns $\pi_1(X)$,
and proves virtual abelianity of the fundamental group of the smooth locus in every
dimension. For the fundamental group of the whole space, the all-dimensional compact
K\"ahler result is already established by Guo-Song-Sturm
\cite[Theorem~1.3]{GSS26}, using the method of~\cite{FGSW26nef}.
Our deduction for the fundamental group of the smooth locus uses the Calabi-Yau theorem
\cite[Theorem~1.3]{FGSW26} through the coregularity reduction.

\subsection{Effective bounds}
The following effective result concerns solvable subgroups:
\begin{introthm}\label{introthm-solv012}
    There are constants $i(c,d)$, depending only on the dimension and coregularity, such that for any projective klt pair $(X,D)$ of dimension $d\ge1$ and absolute coregularity $0 \leq c \leq 2$, the fundamental group $\pi_1^{\rm reg}(X,D)$ has a normal solvable subgroup of derived length at most $2d-1$ and index at most $i(c,d)$.
\end{introthm}

\subsection{Torus covers and augmented irregularity}
The following theorem provides torus covers for pairs:
\begin{introthm}[Torus Covers for pairs]
\label{introthm-torus-covers}
    Let $(X,D)$ be a projective klt pair with $K_X+D \sim_{\mathbb{Q}} 0$. Then there exists a connected finite quasi-\'etale cover $\nu:Y \to X$, such that $Y= A \times Z$ has a product structure with the following properties:
    \begin{enumerate}
        \item $A$ is an Abelian variety of dimension $\tilde{q}(X)$.
        \item There is a boundary $D_Z$ on $Z$ such that $(Z,D_Z)$ is klt and $K_Z+D_Z\sim_{\mathbb{Q}} 0$. Moreover, $\tilde{q}(Z)=0$ and $\nu^*D=\operatorname{pr}_Z^*D_Z$.
    \end{enumerate}
\end{introthm}
The proof, given in Section~\ref{sec:new-results}, uses
\cite[Theorem~4.8]{Amb05}. For pairs, the branched orbifold covers
can contribute additional free abelian factors. In particular, the irregularity of
the underlying variety alone does not determine the virtual rank of $\Gam(X,D)$.
We illustrate this by Example~\ref{ex:rank-product}.
The augmented irregularity is nevertheless bounded by the absolute coregularity:
\begin{introcor}\label{intro:irregularity}
For a projective klt pair $(X,D)$ of finite absolute coregularity,
\[
 \qtilde(X)\le\widehat{\operatorname{coreg}}(X,D).
\]
\end{introcor}

\subsection{Polynomial growth and uniformization}\label{intro:growth-uniformization}
Theorem~\ref{intro:nef} confirms the polynomial-growth conjecture
\cite[Conjecture~1.4]{MW25}, originally formulated in
\cite[Conjecture~1.2]{Wan22}. In fact, it gives an orbifold
strengthening and a bound in terms of absolute coregularity.
\begin{introcor}\label{intro:growth}
Let $(X,D)$ be a projective klt pair with $-(K_X+D)$ nef, and
put $c=\widehat{\operatorname{coreg}}(X,D)$.
Both $\Gam(X,D)$ and $\pi_1(X_{\reg})$ have polynomial growth
of degree at most $2c$.
\end{introcor}

Combining our theorem with the structure theorem of
Matsumura-Wang~\cite[Theorem~1.1]{MW25} gives the following
refinement of uniformization; see Theorem~\ref{thm:nef-uniformization}.
\begin{introthm}\label{intro:nef-uniformization}
Let $(X,D)$ be a projective klt pair with $-(K_X+D)$ nef, and
put $q=\qtilde(X)$. Then $\pi_1(X_{\reg})$ is virtually
$\mathbb Z^{2q}$, and it is finite if and only if $q=0$.
There is a connected finite quasi-\'etale cover $p:X'\to X$
such that
\[
 \pi_1(X'_{\reg})\xrightarrow{\ \simeq\ }\pi_1(X')
 \simeq\mathbb Z^{2q},\qquad
 (\widetilde X',\widetilde D')\simeq
 \mathbb C^q\times Z\times(F,D_F).
\]
Here $\widetilde X'\to X'$ is the ordinary universal cover,
$\widetilde D'$ is the pullback of $D'=p^*D$, the projective
variety $Z$ is canonical with $K_Z\sim0$, and $(F,D_F)$ is a
projective klt pair with $F$ rationally connected and
$-(K_F+D_F)$ nef. Moreover,
$\pi_1(Z_{\reg})=\pi_1(F_{\reg})=1$.
Either factor may be a point.
\end{introthm}

For $D=0$, the product description implies Wang's
uniformization conjecture~\cite[Conjecture~1, first arXiv version]{Wan22}.
The implication from polynomial growth to this decomposition was
already proved in~\cite[Section~6.2]{Wan22}.
The ordinary universal-cover decomposition also follows from
\cite{MW25} together with the recent Calabi-Yau fundamental-group
results~\cite{GSS26,FGSW26}; we do not claim it as an independent
new structure theorem. Related uniformizing covers, without the
assertion that they are universal, are constructed in the more
general K\"ahler setting in~\cite[Corollary~1.3]{MWWZ26}.
The additional conclusion here controls the fundamental group of the smooth loci,
including those of the factors, and gives their rank and growth.

\subsection{Relation to our earlier manuscript on coregularity and fundamental groups}\label{sec:earlier-paper}
This article develops further the results of our earlier preprint
\emph{Fundamental Groups, Coregularity, and Low Dimensional Klt
Calabi Yau Pairs}~\cite{BF24}.
The principal extension is the all-dimensional coregularity theorem
and its application to projective klt pairs with nef log-anticanonical
divisor, building upon~\cite[Theorem~1.3]{FGSW26}, combined with
the log index-one cover of~\cite[Proposition~12]{CGG24} and our modified reduction statement Theorem~\ref{thm-induction-coreg}. We would like to mention that the preprint~\cite{FGSW26} also announces a forthcoming separate treatment of pairs, nef
varieties, and compact K\"ahler spaces. 

Theorem~\ref{thm-induction-coreg} of the present work is a modified version of~\cite[Theorem~4.2]{BF24}. The conjectural and low-dimensional virtual-abelianity statements
of~\cite{BF24} are subsumed by the theorems here. We do not use its
passage from canonical threefolds to terminalizations as such a
birational modification need not preserve the regular-locus
fundamental group. The associated Euler-characteristic argument is
therefore replaced by the direct application of~\cite{FGSW26}.

We remark in particular that for pairs, the rank is bounded by twice the absolute coregularity;
our earlier bound $\max(4,2\qtilde(X))$ from~\cite{BF24} in dimension three is not
valid, as Example~\ref{ex:rank-product} shows. In the present work, for standard coefficients,
Proposition~\ref{prop:cy-cover} instead identifies the exact rank
using the augmented irregularity of the log index-one cover of~\cite{CGG24}.

\subsection*{Acknowledgements}

The authors would like to thank Joaqu\'in Moraga for suggesting the problem and useful discussions,
as well as Stefano Filipazzi, Stefan Kebekus, and Mirko Mauri for very useful comments and references, in particular for pointing out several issues in the earlier preprint~\cite{BF24}.
FF thanks UCLA for its hospitality during his visit.

\subsection*{AI disclosure}
All original mathematical ideas and proofs in this manuscript are the authors' own. Thus, we take full responsibility for the content and mathematical correctness of the manuscript. ChatGPT 6 Astra was used only for an additional check of the arguments, proofreading and improvement of the exposition. This has not led to substantial changes, in particular not with respect to mathematical content.

\section{Preliminaries}

In this section, we provide the necessary background and prove some preliminary results for the rest of the article.

\subsection{Virtual abelianity and solvability of groups}
\begin{definition}
\label{def:virtually-abelian}
We say that a group $G$ is \textit{virtually abelian of rank $n$} if it has a subgroup of finite index which is abelian of rank $n$.
\end{definition}

\begin{remark}
A group $G$ being virtually abelian of rank $n$ is equivalent to having a \textit{normal} subgroup of finite index which is abelian of rank $n$, since one can take the normal core of the abelian subgroup from Definition~\ref{def:virtually-abelian}, which is again abelian of the same rank and has finite index in $G$.

All groups occurring here are indeed finitely generated: a complex algebraic
variety has the homotopy type of a finite CW complex, and orbifold fundamental
groups are quotients of fundamental groups of algebraic open sets. 

For these groups, the above definition is equivalent to having a finite-index subgroup $\mathbb Z^n$.
\end{remark}

\begin{lemma}
\label{lem-virtual-groups}
Let $F, G$ and $H$ be groups fitting in an exact sequence 
$$F \xrightarrow{\phi} G \xrightarrow{\psi} H\rightarrow 1.$$ 
If $H$ is finitely generated and virtually abelian of rank $n$ and $F$ is finite, then $G$ is virtually abelian of rank at most $n$.
\end{lemma}

\begin{proof}
By replacing $F$ with its image $\phi(F)$ we can assume that $F$ is a subgroup of $G$ with $f$ elements.
Since $H$ is virtually abelian, there exists an abelian group $A \leqslant H$ of rank $n$ and of finite index in $H$. After replacing $A$ by a subgroup of finite index, we may assume that $A\simeq\mathbb Z^n$. Let $A$ be generated by $\{a_i\}_{\{1 \leq i\leq n\}}$. Let $g_i$ be elements in $G$ mapping to $a_i$ via $\psi$. For any $x$ in $F$ and $g_i$, we have that $g_i x g_i^{-1}$ gets mapped to the identity $e\in H$ via $\psi$, hence $g_i x g_i^{-1} \in F$. Thus conjugation by elements  $y \in \langle g_i \rangle$ gives a group action  on $F$. Since $F$ is a finite set, there exists a uniform constant $m$ such that $m$-iterations of all of these actions are the identity, in particular $y^{m} x y^{-m} =x$ for any $x \in F$ and $y \in \langle g_i \rangle$. 

On the other hand, for any $y \in \langle g_i \rangle$, there exists $x\in F$, such that $g_i^{m} y g_i^{-m}=x y$, since it maps to the same element of $A$ as $y$. Therefore 
$$
g_i^{fm}g_j^{fm}g_i^{-fm}=g_i^{(f-1)m}xg_j^{fm}g_i^{-(f-1)m}=xg_i^{(f-1)m}g_j^{fm}g_i^{-(f-1)m}= \ldots=x^{f}g_j^{fm}=g_j^{fm}.
$$
We conclude that $\langle  g_i^{fm}\rangle_i$ is an abelian subgroup of rank $n$ of $G$. This subgroup has finite index in $G$, as it surjects onto $\langle a_i^{fm}\rangle_i$, a finite index subgroup of $H$, via a homomorphism with finite kernel.

\end{proof}

\begin{definition}
A group is virtually solvable of derived length at most $\ell$ and index
at most $i$ if it has a solvable subgroup of derived length at most
$\ell$ and index at most $i$.
\end{definition}

\begin{lemma}\label{lem-fin-sol}
    Let $F, G$ and $H$ be groups fitting in an exact sequence 
$$F \xrightarrow{\phi} G \xrightarrow{\psi} H\rightarrow 1.$$ 
If $H$ is virtually solvable of derived length at most $n$ and index at most $i$ and $F$ is finite, then $G$ is virtually solvable of derived length at most $n+1$ and index at most $i(|F|-1)!$.
\end{lemma}

\begin{proof}
Let $H'$ be the solvable subgroup of $H$ of derived length at most $n$ and index at most $i$. By letting $G'$ be the preimage of $H'$ under $\psi$, we obtain the exact sequence
$$F \xrightarrow{\phi} G' \xrightarrow{\psi} H'\rightarrow 1,$$ 
where $[G:G']=[H:H']\le i$.

By taking a quotient if needed, we may assume that $F$ is a subgroup of $G'$.
Let $Z$ be the centralizer of $F$ in $G'$. The index of $Z$ in $G'$ is at most $|{\rm Aut}(F)|$, which is at most $(|F|-1)!$.

Since $F\cap Z$ is abelian and $\psi(Z) \leqslant H'$ is solvable of derived length at most $n$, the group $Z$ is solvable of derived length at most $n+1$. Hence $G$ is virtually solvable of derived length at most $n+1$ and index at most $i|{\rm Aut}(F)|\leq i(|F|-1)!$.

\end{proof}

The following was proven in~\cite{CD89}, we recall the argument for the sake of completeness.

\begin{lemma}\label{lem:characteristic}
If a finite group $F$ has an abelian subgroup of index at most $J$,
then it has a characteristic abelian subgroup of index at most $J^2$. 
\end{lemma}
\begin{proof}
For $H\le F$ put
$m(H)=|H|\,|C_F(H)|$. For subgroups $H,K\le F$, counting the
products $HK$ and $C_F(H)C_F(K)$ gives
\[
 m(H)m(K)\le m(H\cap K)m(\langle H,K\rangle).
\]
Thus the subgroups maximizing $m$ are closed with respect to intersection.
Their unique smallest member $A$ is characteristic. The centralizer
of a maximizing subgroup is also maximizing, since
$H\subseteq C_F(C_F(H))$. Hence $A\subseteq C_F(A)$, so $A$ is
abelian. If $B$ is the given abelian subgroup, then
\[
 |A|\,|F|\ge m(A)\ge m(B)\ge |B|^2\ge |F|^2/J^2.
\]
This proves the assertion.
\end{proof}

\begin{lemma}\label{lem-virt-sol}
Suppose $F\to G\to H\to1$ is exact, $F$ is finite and has an
abelian subgroup of index at most $J$, and $H$ has a solvable
subgroup of derived length at most $\ell$ and index at most $i$.
Then $G$ has a solvable subgroup of derived length at most $\ell+2$
and index at most $i(J^2)!$.
\end{lemma}
\begin{proof}
Replace $F$ by its image; the same index bound persists. By
Lemma~\ref{lem:characteristic}, choose a characteristic abelian
subgroup $A\le F$ with $[F:A]\le J^2$. It is normal in $G$.
Apply Lemma~\ref{lem-fin-sol} to
$F/A\to G/A\to H\to1$, and lift the resulting solvable subgroup
through the abelian kernel $A$.
\end{proof}
Taking a normal core does not increase derived length and changes an index
bound $i$ to at most $i!$; in particular, if $[G:H]\le i$, the action of $G$ on $G/H$
has kernel $\bigcap_{g\in G}gHg^{-1}\le H$ and image of order at most $i!$.

\subsection{Generalized pairs and singularities of the Minimal Model Program}
A generalized pair $(X,B,\bM)$ consists of a normal variety, an
effective $\mathbb Q$-divisor $B$, and a b-$\mathbb Q$-Cartier
b-divisor $\bM$ which descends to a nef divisor on a projective
birational model, with $K_X+B+\bM_X$ $\mathbb Q$-Cartier.
For $f:V\to X$ write
\[
 K_V+B_V+\bM_V=f^*(K_X+B+\bM_X).
\]
The generalized log discrepancy of a divisor $E$ on $V$ is
$1-\coeff_E B_V$. Analogous to the non-generalized notions, \textit{generalized klt} and \textit{generalized lc} mean that these log discrepancies are strictly positive and nonnegative, respectively.
Generalized dlt modifications are taken in the sense of~\cite{FS23}.

If $X$ is $\mathbb Q$-factorial and $(X,B,\bM)$ is generalized klt,
then $(X,B)$ is ordinary klt: on a model where $\bM$ descends,
the negativity lemma gives $f^*\bM_X-\bM_V\ge0$, so the ordinary
boundary is bounded above by the generalized boundary.

\begin{lemma}\label{lem:interpolation}
Let $(X,P)$ be a $\mathbb Q$-factorial klt pair and let
$(X,B,\bM)$ be generalized lc, with $P\le B$.
For rational $0<t\le1$, the generalized pair
\[
 (X,(1-t)B+tP,(1-t)\bM)
\]
is generalized klt. On any higher model its generalized log
discrepancies are
\[
 (1-t)A_{X,B,\bM}(E)+tA_{X,P}(E)>0.
\]
In particular, for any fixed finite collection of lc places
of $(X,B,\bM)$, the pair $(X,(1-t)B+tP,(1-t)\bM)$  has discrepancies in $(0,1)$ for $t>0$ sufficiently
small.
\end{lemma}
\begin{proof}
Pull back the identity
\[
 K_X+(1-t)B+tP+(1-t)\bM_X
 =(1-t)(K_X+B+\bM_X)+t(K_X+P).
\]
On a common resolution the resulting boundary is $(1-t)B_V+tP_V$,
where $K_V+P_V=f^*(K_X+P)$. Its log discrepancies are therefore
the stated convex combinations. For an lc place they equal
$tA_{X,P}(E)$, which belongs to $(0,1)$ when $t>0$ is sufficiently small.
\end{proof}

\begin{lemma}\label{lem-big-klt}
Let $(X,B,\bM)$ be a projective generalized klt pair and let $H$
be an ample $\mathbb Q$-Cartier divisor. For every rational $\epsilon>0$
there is an ordinary klt boundary $\Delta\ge B$ with
\[
 K_X+\Delta\sim_{\mathbb Q}K_X+B+\bM_X+\epsilon H.
\]
If finitely many prescribed prime divisors $E_i$ over $X$ satisfy
$0<A_{X,B,\bM}(E_i)<1$, we may also arrange
$0<A_{X,\Delta}(E_i)<1$ for every $i$.

\end{lemma}
\begin{proof}
Take a log resolution $f:V\to X$ on which $\bM$ descends and all
the prescribed divisors occur. Write
$K_V+B_V+\bM_V=f^*(K_X+B+\bM_X)$.
The divisor $\bM_V+\epsilon f^*H$ is nef and big. By Kodaira's lemma write it
as $H_\delta+\delta E$, up to $\mathbb Q$-linear equivalence,
where $E\ge0$ is fixed and $H_\delta$ is ample for sufficiently small rational
$\delta>0$. Choose a general effective representative $G_\delta\sim_{\mathbb Q}H_\delta$
with sufficiently small coefficients. We may arrange that
$(V,B_V+G_\delta+\delta E)$ is sub-klt and that $G_\delta$ contains
none of the prescribed divisors. Their boundary coefficients tend to
those of $B_V$ as $\delta\to0$ and therefore remain in $(0,1)$.
Set $\Delta=B+f_*(G_\delta+\delta E)\ge B$.
Then $K_X+\Delta\sim_{\mathbb Q}K_X+B+\bM_X+\epsilon H$ is
$\mathbb Q$-Cartier, and
\[
 K_V+B_V+G_\delta+\delta E=f^*(K_X+\Delta).
\]
Indeed, the difference is exceptional and $\mathbb Q$-linearly trivial
over $X$, so it vanishes by the negativity lemma.
The equality proves both that $(X,\Delta)$ is klt and that
$A_{X,\Delta}(E_i)=1-\coeff_{E_i}(B_V+G_\delta+\delta E)$
belongs to $(0,1)$. Negative components of $B_V$ are exceptional
and disappear under $f_*$; the added divisor is effective.
\end{proof}

\begin{lemma}\label{lem-BCHM-generalized-pairs}
In Lemma~\ref{lem:interpolation}, assume in addition that $X$ is projective.
For any finite collection $E_1,\ldots,E_r$ of exceptional prime divisors
over $X$ with $A_{X,B,\bM}(E_i)=0$, there is a projective birational
morphism $\mu:Y\to X$, with $Y$ $\mathbb Q$-factorial, whose
exceptional prime divisors are exactly $E_1,\ldots,E_r$.
\end{lemma}
\begin{proof}
Choose $t>0$ sufficiently small in Lemma~\ref{lem:interpolation}
so that all the prescribed generalized log discrepancies belong to $(0,1)$.
Apply Lemma~\ref{lem-big-klt} to this generalized klt pair, using an
ample divisor on $X$, to obtain an ordinary klt boundary $\Delta$
with $0<A_{X,\Delta}(E_i)<1$ for every $i$.
Now~\cite[Corollary~1.4.3]{BCHM10} extracts exactly these divisors
on a $\mathbb Q$-factorial model. Its discrepancy hypotheses hold
because $(X,\Delta)$ is klt and all the specified discrepancies are smaller than one.
\end{proof}

\begin{proposition}[Generalized canonical bundle formula]\label{prop:cbf}
Let $f:Z\to W$ be a projective contraction and let $(Z,B,\bM)$
be generalized lc, with $K_Z+B+\bM_Z\sim_{\mathbb Q}0$.
There is an induced generalized lc pair $(W,B_W,\bN)$ with
$B_W\ge0$, $\bN$ b-nef and b-$\mathbb Q$-Cartier, and
$K_W+B_W+\bN_W\sim_{\mathbb Q}0$.
The coefficient of a prime divisor $E$ in $B_W$ is $1-t_E$,
where $t_E$ is the generalized lc threshold of $f^*E$ over its
generic point. If the total pair is generalized klt, so is the base.
\end{proposition}
This is~\cite[Theorem~1.4 and Proposition~4.16]{Fil20}.
Note that for a general projective surjective morphism, Stein factorization
and the generically finite formula of~\cite[Theorem~1.2]{HL21}
give the corresponding generalized lc base structure. We use the
contraction case below.

\begin{lemma}\label{lem-image-lcc}
In the setting of Proposition~\ref{prop:cbf}, the image of a generalized lc
center is a generalized lc center of the base, or it is the whole base.
\end{lemma}
\begin{proof}
Let $C$ have proper image $T$, and suppose $T$ is not a generalized
lc center of the base. Near the generic point of $T$, choose a
general effective Cartier divisor containing $T$ but containing
no generalized lc center of the base. Such a choice is possible:
centers properly contained in $T$ can be discarded on this open
set, and centers properly containing $T$ need not lie in the divisor.
On a common log resolution this divisor has zero order along every
discrepancy-zero divisor, since it contains no generalized lc center.
A sufficiently small positive multiple can therefore be added while
preserving generalized lc singularities. Adding its pullback upstairs makes the discrepancy
of any lc place centered on $C$ negative. The discriminant changes
by exactly that divisor and the moduli part is unchanged. This
contradicts the lc equivalence in~\cite[Proposition~4.16]{Fil20}.
\end{proof}

\subsection{Coregularity}
\begin{definition}\label{def:absolute-coregularity} 
Let $(X,B, {\bf M})$ be a generalized log canonical pair. Let $g:(Y,B_Y, {\bf M}) \rightarrow (X,B, {\bf M})$ be a $\mathbb{Q}$-factorial gdlt modification, with $\lfloor B_Y \rfloor=E_1+\ldots +E_r$. 
We define $\mathcal{D}(Y,B_Y,{\bf M})$, the \textit{dual complex} of $(Y,B_Y,{\bf M})$, to be:

\begin{enumerate}
    \item For each $E_i$, there is a vertex $s_{E_i}$ in $\mathcal{D}(Y,B_Y)$,
    \item For any subset $I \subseteq \{1,\ldots, r\}$, and any irreducible component $Z$ of $\cap_{i \in I} E_i$, there is a $(|I|-1)$-dimensional simplex $s_Z$.
    \item For any 
    \begin{itemize}
        \item $I\subseteq \{1, \ldots r\}$,
        \item  $j \in I$,
        \item  any irreducible component $Z$ of $\cap_{i \in I} E_i$,
        \item and $W$ the unique irreducible component of $\cap_{i \in I\setminus \{j\}} E_i$ containing $Z$,
    \end{itemize}
    there is a gluing map given by the inclusion of $s_W$ into $s_Z$ as the unique face that does not contain the vertex $s_{E_j}$.
\end{enumerate}
We define the \textit{dimension of the dual complex} to be the smallest dimension of a maximal simplex in $\mathcal{D}(Y,B_Y,{\bf M})$ (with respect to inclusion). The empty complex has dimension $-1$ by definition. By~\cite{FMP22}, the dimension of $\mathcal{D}(Y,B_Y, {\bf M})$ does not depend on the gdlt modification.
Hence, we can define the \textit{coregularity} of a generalized pair $(X,B, {\bf M})$, to be:

$$\text{coreg}(X,B, {\bf M}):= \text{dim}\, X-\text{dim}\, \mathcal{D}(Y,B_Y, {\bf M})-1.$$

For a generalized pair $(X,D, {\bf M})$ admitting a generalized log Calabi-Yau completion, we define the \textit{absolute coregularity} ${\rm \hat{coreg}}(X,D, {\bf M})$ to be the smallest value of $\text{coreg}(X,B,{\bf M+ N})$, where $(X,B, {\bf M+N})$ is generalized log Calabi-Yau, with $B\geq D$ and $\mathbf N$ a nef b-divisor. 
\end{definition}

A generalized pair admitting such a completion is said to be of
generalized log Calabi-Yau type. 

The absolute coregularity as defined here may be lower than the one in
\cite[Definition~2.24]{FMP22}, where the nef b-divisor is held fixed.
We need to use this definition for the induction in the proof of the main theorems.

\begin{remark}\label{rem:coreg-facts}
The coregularity of a projective generalized lc Calabi-Yau pair is
invariant under crepant birational changes. Its dual complex is
equidimensional, and on a generalized dlt model all minimal lc
centres have dimension equal to the coregularity. For a generalized
klt pair we regard the whole variety as the sole centre.
For the Fano type fibrations with compatible generalized dlt models
used below, the coregularity of the induced base is at most that of
the total space.
\end{remark}
\begin{proof}
The first assertion follows from the invariance of the dual
complex, and equidimensionality from~\cite[Theorem~1.6]{FS23}; see
also~\cite{FMP22}. On compatible generalized dlt models, take a
minimal centre upstairs and apply Lemma~\ref{lem-image-lcc}.
Its image either is the base, of no larger dimension, or contains
a minimal centre of the base. Equidimensionality proves the last claim.
\end{proof}

\begin{lemma}\label{lem-coreg-bir-morphism}
Let $f:Y\dashrightarrow X$ be a birational contraction of generalized
pairs, with $K_Y+D_Y+\bM_Y=f^*(K_X+D_X+\bM_X)$ (interpreted
on a common resolution), and assume $D_Y\ge0$.
Then
\[
 \widehat{\operatorname{coreg}}(Y,D_Y,\bM)
 \ge\widehat{\operatorname{coreg}}(X,D_X,\bM).
\]
\end{lemma}
\begin{proof}
Push any generalized lc Calabi-Yau completion on $Y$ to $X$.
Its boundary still dominates $D_X$, and its nef b-divisor is unchanged.
The completion is crepant and has the same coregularity. Taking the
minimum proves the assertion. This is used only when the indicated
boundary is effective.
\end{proof}

For a klt Calabi-Yau pair, the fixed-nef-part version of the following
coregularity assertion is contained in~\cite[Remark~2.26]{FMP22}.
We also allow the nef part of the completion to vary.
\begin{lemma}\label{lem:cy-coreg}
Let $(X,P)$ be a projective klt pair admitting a generalized lc
Calabi-Yau completion. If $K_X+P$ is pseudo-effective, then
$K_X+P\sim_{\mathbb Q}0$ and
$\widehat{\operatorname{coreg}}(X,P)=\dim X$.
\end{lemma}
\begin{proof}
The case $\dim X=0$ is immediate. Pass to a small $\mathbb Q$-factorialization;
completions pull back and push forward crepantly with effective boundaries,
so absolute coregularity is unchanged. For any completion
$(X,B,\bM)$ we have
$-(K_X+P)\sim_{\mathbb Q}B-P+\bM_X$.
The trace of a b-nef divisor is pseudo-effective. The pseudo-effective
cone contains no line, so $K_X+P\num0$ and
$B-P+\bM_X\num0$. Since both summands are pseudo-effective,
they are numerically trivial. The effective divisor $B-P$ has zero
intersection with an ample complete-intersection curve, hence $B=P$;
also $\bM_X\num0$.
On a model $f:V\to X$ where $\bM$ descends,
$E=f^*\bM_X-\bM_V\ge0$ is exceptional by negativity.
If $E\ne0$, choose an ample divisor $H$ on $V$ and put $n=\dim V$.
Then
\[
 \bM_V\cdot H^{n-1}=-E\cdot H^{n-1}<0,
\]
contradicting nefness. Thus we get $E=0$.

The nef part descends, and the completion has the ordinary klt
discrepancies. Its dual complex is empty. This holds for every
completion, proving the coregularity claim. Finally, numerical-triviality
abundance~\cite[Theorem~1.2]{Gon13} gives $K_X+P\sim_{\mathbb Q}0$.
\end{proof} 

\subsection{Orbifold fundamental groups}
\begin{definition}\label{def:standard}
A divisor $D$ is said to have \textit{standard} coefficients, when the coefficients are in the set $\{1-\frac{1}{n}\}_{\{n \in \mathbb{N}\}}\cup \{1\}$.
For an effective $\mathbb{Q}$-divisor $D$, the \textit{standard approximation} $D_{\text{st}}$ is the largest effective divisor with standard coefficients, such that $D_{\text{st}}\leq D$.
\end{definition}
\begin{definition}\label{def:orb}
For a generalized  pair $(X,D,{\bf M})$ with $D_\text{st}=\sum_{j \in J} D_j + \sum_{i \in I} (1-\frac{1}{n_i})D_i$, we define the \textit{orbifold fundamental group of the smooth locus} $\pi_1^{\text{reg}}(X,D,{\bf M})$ to be 
$$
\pi_1(X_{\text{reg}}\setminus \text{supp}(D_{\text{st}}))/N,
$$
where $X_{\text{reg}}$ is the smooth locus of $X$ and $N$ is the normal subgroup of $\pi_1(X_{\text{reg}}\setminus \text{supp}(D_{\text{st}}))$ generated by $\gamma_i^{n_i}$, where $\gamma_i$ is a loop around $D_i$ for $i \in I$. 
\end{definition}

We will use the convention that $1=1-\frac{1}{\infty}$, thus for divisors $D_i$ with coefficient $1$, the loop $\gamma_i$ will not be a generator of $N$, in alignment with the previous definition.

\begin{theorem}[\cite{Bra21,BFMS20}]
\label{thm:fano}\label{thm-Bra21-gen-pair}
For each integer $r\ge1$ there is a positive integer $J(r)$ with the following property.
Let $(V,C,\bM)$ be a projective $\mathbb Q$-factorial generalized
klt pair of dimension $r$ such that $-(K_V+C+\bM_V)$ is nef and
big, and let $R\le C$ be an effective boundary with standard coefficients.
Then $\Gam(V,R)$ is finite and has an abelian subgroup of index
at most $J(r)$.
\end{theorem}
\begin{proof}
Put $L=-(K_V+C+\bM_V)$ and write $L\sim_{\mathbb Q}A+E$ with
$A$ ample and $E\ge0$. Choose $\epsilon>0$ small enough that
$(V,C+\epsilon E,\bM)$ is generalized klt. Ordinary approximation
with the ample summand $\epsilon A$ gives a klt boundary
$\Delta\ge C+\epsilon E\ge R$ with
$-(K_V+\Delta)\sim_{\mathbb Q}(1-\epsilon)L$ nef and big.
An additional small Kodaira perturbation makes this divisor ample, increases $\Delta$ and preserves klt singularities.
Thus $(V,R)$ is of log Fano type. Finiteness follows from
\cite[Theorem~2]{Bra21}, and the uniform index bound from
\cite[Theorem~3]{BFMS20}.
\end{proof}

For curves, we get the following:

\begin{proposition}\label{prop-dim1}
For projective generalized klt curve pairs of finite absolute
coregularity, coregularity zero implies a finite regional orbifold
group, and coregularity one implies virtual abelianity of rank at
most two. These groups have abelian subgroups of index at most $60$.
\end{proposition}
\begin{proof}
On a smooth curve the nef part has nonnegative degree. A completion
therefore forces either an elliptic curve with zero boundary and
degree-zero nef part, or $\mathbb P^1$ with $\deg D_{\st}\le2$.
In coregularity zero the completion has a coefficient-one point;
the klt boundary has strictly smaller degree, so
$\deg D_{\st}<2$. 

Write $D_{\st}=\sum_{i=1}^r(1-1/m_i)p_i$, with $m_i\ge2$.
The orbifold Euler characteristic on $\mathbb P^1$ is
$2-\deg D_{\st}$. If it is positive and $r\ge3$, the orbifold is
spherical: its signature $(m_1,\ldots,m_r)$, the list of orders at the points, is $(2,2,m)$, $(2,3,3)$, $(2,3,4)$ or $(2,3,5)$.
The corresponding groups are dihedral, tetrahedral, octahedral or
icosahedral. With at most two cone points, the presentation directly
gives a finite cyclic group (possibly trivial). In all these cases it is well known that
there is an abelian subgroup of index at most $60$.

If $\deg D_{\st}=2$, the orbifold Euler characteristic is zero,
and the Euclidean signatures are $(2,2,2,2)$, $(2,3,6)$,
$(2,4,4)$ and $(3,3,3)$; see~\cite[Proposition~5.3]{GLM23}.
The corresponding orbifold fundamental groups contain their translation
subgroups $\mathbb Z^2$ as normal subgroups of indices
$2,6,4,3$, respectively.
In the elliptic-curve case the boundary is zero, so the regional
orbifold fundamental group is the ordinary fundamental group of the
elliptic curve, namely $\mathbb Z^2$.
\end{proof}

For surfaces on the other hand, we have:

\begin{proposition} \label{prop-dim2}
    Let $(X,D)$ be a projective klt Calabi-Yau pair of dimension $2$, with standard coefficients, then $\pi_1^{\rm reg}(X,D)$ is virtually abelian of rank at most $4$. It has an abelian subgroup of index at most $3840$.
\end{proposition}
\begin{proof}
By~\cite[Theorem~B]{GLM23}, $\pi_1^{\rm reg}(X,D)$ is finite of order at most
$3840$, or has a subgroup $\mathbb Z^4$ of index at most $96$.
\end{proof}

\section{Fundamental Groups under Morphisms}
In this section we prove some results regarding the behavior of the fundamental group after performing certain birational maps and fibrations.

\begin{lemma}\label{lem-MMP-group}
Let $\phi:X\dashrightarrow X'$ be a birational contraction of
normal varieties, and let
$P'=\phi_*P$. Then
\[
 \Gam(X',P')\twoheadrightarrow\Gam(X,P).
\]
If $\phi$ is small, the groups are isomorphic.
\end{lemma}
\begin{proof}
On the target, the complement of the common isomorphic open set has codimension at least two. 
On the source, for the contracted divisors, we get loop relations, including the
prescribed orbifold relations. 
The boundary coefficients agree on uncontracted divisors. For a small map both complements have
codimension at least two, so no additional relations occur.
\end{proof}

\begin{lemma}\label{lem-dlt-group}
Let $(X,P)$ be a projective $\mathbb Q$-factorial klt pair with
standard coefficients and let $(X,B,\bM)$ be a generalized lc
Calabi-Yau completion. Suppose $g:Y\dashrightarrow X$ is birational and every prime
divisor on $Y$ that is exceptional over $X$ is a generalized lc place of this
completion. 

For every extracted
divisor $E_i$, one can choose an arbitrarily large divisible integer
$q_i$ such that, with
\[
 R_Y=P_Y^{\mathrm{str}}+\sum_i(1-1/q_i)E_i,
\]
there is a surjection $\Gam(Y,R_Y)\twoheadrightarrow\Gam(X,P)$.
Here, $P_Y^{\mathrm{str}}$ includes the transforms of the components
of $P$ which survive on $Y$.
\end{lemma}
\begin{proof}
By Lemma~\ref{lem-BCHM-generalized-pairs}, take a projective
extraction $h:Y_0\to X$ of exactly the divisors extracted by $g$.
The map $Y_0\dashrightarrow Y$ extracts no divisors. On the
complement of the exceptional locus, the fundamental groups agree with the
corresponding fundamental group on $X$ after removing codimension-two subsets.
The image in $\Gam(X,P)$ of a loop around each $E_i$ has finite
order: it is the image of a regional local orbifold fundamental group of the
klt pair $(X,P)$, which is finite by~\cite[Theorem~1]{Bra21}.
Choose $q_i$ divisible by this order. Imposing the $q_i$-th power
relations gives a surjection from $\Gam(Y_0,R_{Y_0})$.
Lemma~\ref{lem-MMP-group} then gives the surjection from $Y$.
\end{proof}

\begin{lemma}\label{lem-Fano-extraction}
Let $(F,P_F)$ be a projective $\mathbb Q$-factorial klt pair such that
$-(K_F+P_F)$ is big and nef, and let $(F,B_F,\bM)$ be a generalized
Calabi-Yau completion. Let $\phi:F_0\to F$ be a $\mathbb Q$-factorial
extraction of generalized log canonical places of $(F,B_F,\bM)$.
Define $R_{F_0}=\phi^{-1}_*(P_F)+E$, where $E\ge0$ is a
$\phi$-exceptional divisor with standard coefficients smaller than one,
and assume $P_F$ has standard coefficients.
Then $\Gam(F_0,R_{F_0})$ is finite and has an abelian subgroup of
index at most $J(\dim F)$.
\end{lemma}
\begin{proof}
Let $B_{F_0}$ and $P_{F_0}$ be defined by the equalities
\[
 \phi^*(K_F+B_F+\bM_F)=K_{F_0}+B_{F_0}+\bM_{F_0},
 \qquad \phi^*(K_F+P_F)=K_{F_0}+P_{F_0}.
\]
For any $0<\alpha<1$, the divisor
$-K_{F_0}-(1-\alpha)B_{F_0}-\alpha P_{F_0}-(1-\alpha)\bM_{F_0}$
is big and nef. Indeed,
\[
 -\bigl(K_{F_0}+(1-\alpha)B_{F_0}+\alpha P_{F_0}
 +(1-\alpha)\bM_{F_0}\bigr)
 \sim_{\mathbb Q}-\alpha\phi^*(K_F+P_F).
\]
The sub-pair $(F_0,P_{F_0})$ is sub-klt. Since $\phi$ only extracts
glc places of $(F,B_F,\bM)$, the pair $(F_0,B_{F_0},\bM)$ is
generalized log canonical with $B_{F_0}\ge P_{F_0}$.
Hence for $\alpha$ close enough to zero, the pair
$(F_0,(1-\alpha)B_{F_0}+\alpha P_{F_0},(1-\alpha)\bM)$ is gklt
by pulling back the interpolation of Lemma~\ref{lem:interpolation},
is weakly Fano, and has
$(1-\alpha)B_{F_0}+\alpha P_{F_0}\ge R_{F_0}$.
Note that every exceptional coefficient of $B_{F_0}$ is one, so any fixed
exceptional coefficients smaller than one are allowed.
Therefore by Theorem~\ref{thm-Bra21-gen-pair} we get the finiteness claim and the
stated uniform index bound.
\end{proof}

\begin{lemma}\label{lem-dlt-mod-Fano-fibration}
Let $f:Z\to W$ be a Fano type fibration. Let $(Z,B,\bM)$ be a
generalized log Calabi-Yau pair and $(W,B_W,\bN)$ the generalized
pair induced by the canonical bundle formula. Let $W'\to W$ be a
$\mathbb Q$-factorial generalized dlt modification of $(W,B_W,\bN)$.
Then there exists a $\mathbb Q$-factorial generalized dlt model
$(Z',B',\bM)$ and a birational contraction $g:Z'\dashrightarrow Z$,
whose exceptional divisors are generalized log canonical places of
$(Z,B,\bM)$, together with a Fano type fibration $f':Z'\to W'$
making the following diagram commute:
\[
\begin{tikzcd}
 (Z,B,\bM)\arrow[d,"f"] &
 (Z',B',\bM)\arrow[l,dashed,"g"']\arrow[d,"f'"]\\
 (W,B_W,\bN) & (W',B_{W'},\bN)\arrow[l].
\end{tikzcd}
\]
Here, for every prime divisor of $Z$, there is a prime strict transform
on $Z'$. On general fibres $F'$ and $F$, every divisor of $F$
therefore survives on $F'$, and the additional divisors on $F'$
are lc places of the restricted completion.
\end{lemma}
\begin{proof}
We use the first stage of the construction in
\cite[Lemma~2.36]{Mor21a}. Start with a log resolution $p:Z_0\to Z$
dominating $W'$, on which $\bM$ descends. Write
$K_{Z_0}+B_0+\bM_{Z_0}=p^*(K_Z+B+\bM_Z)$, and let $B_0^+$
be the strict transform of $B$, plus the reduced $p$-exceptional divisor.
Now let
\[
 E:=B_0^+-B_0=\sum_{T\text{ exceptional}/Z}
 A_{Z,B,\bM}(T)T\ge0,\qquad
 K_{Z_0}+B_0^++\bM_{Z_0}\sim_{\mathbb Q}E.
\]
The relative MMP over $W'$ in that construction contracts
$\operatorname{Supp}E$ and produces a generalized dlt model
of Fano type over $W'$. At each step its MMP divisor is
$\mathbb Q$-linearly equivalent to the effective transform $E_i$.
A divisorial contraction cannot contract a divisor outside
$\operatorname{Supp}E_i$: a contracted curve through a general point
of that divisor is not contained in $\operatorname{Supp}E_i$, and
hence has nonnegative intersection with $E_i$, contrary to negativity
of the contracted ray. Thus every
original divisor of $Z$ survives, and every remaining divisor
exceptional over $Z$ has discrepancy zero.

We do not perform the additional modifications arranging the inverse
image of the non-klt locus as in~\cite[Lemma~2.36]{Mor21a}.

Finally, restricting to a sufficiently general fibre gives the last assertion. 
This can be checked on the same resolution, where horizontal
exceptional divisors restrict to exceptional divisors and keep
their discrepancy coefficients. All remaining exceptional and
indeterminacy loci meet a sufficiently general fibre in codimension
at least two.
\end{proof}

\begin{lemma}\label{lem-fibration}
Let $f:Z\to W$ be a projective contraction of normal varieties and
let $R$ be a boundary with standard coefficients smaller than one.
For a prime divisor $E$ on $W$, write, over its generic point,
\[
 f^*E=\sum_j m_{E,j}T_{E,j},\qquad
 \coeff_{T_{E,j}}R=1-1/n_{E,j}.
\]
Set $k_E=\gcd_j(m_{E,j}n_{E,j})$ and
$R_W=\sum_E(1-1/k_E)E$. If $T_{E,j}$ is not a component of
$\operatorname{Supp}R$, its coefficient in $R$ is zero, so $n_{E,j}=1$.
There is a smooth open subset $V\subset W$ with complement of
codimension at least two such that, for $U=f^{-1}(V)$ and a general
fibre $F$, we have an exact sequence
\[
 \Gam(F,R|_F)\longrightarrow\Gam(U,R|_U)
 \longrightarrow\Gam(W,R_W)\longrightarrow1.
\]
Moreover, $\Gam(U,R|_U)$ surjects onto $\Gam(Z,R)$.
\end{lemma}
\begin{proof}
Remove from the base its singular locus and the images of those prime
divisors on $Z$ whose images in $W$ have codimension at least two. Denote the resulting big open subset by $V$ and put
$U=f^{-1}(V)$. By~\cite[Lemma~1.5.A]{Nor83}, we can choose a smaller
nonempty open set $V^\circ\subset V$, with divisorial complement, such that
\[
 U^\circ:=f^{-1}(V^\circ)_{\reg}\setminus\Supp R\longrightarrow V^\circ
\]
is a topological fibre bundle. By generic flatness and generic smoothness,
we may assume its connected fibre is $F_{\reg}\setminus\Supp(R|_F)$,
and that $V^\circ$ avoids the images of the vertical boundary components.

Now, the loops of horizontal boundary components are represented, up to
conjugation, by loops in a general fibre. 
Quotienting the homotopy sequence by their prescribed orbifold relations 
therefore gives an exact sequence
\[
 \Gam(F,R|_F)\xrightarrow{\psi_1}H\xrightarrow{\psi_2}Q\longrightarrow1,
\]
where $H=\pi_1(U^\circ)/N_{\mathrm{hor}}$ and $Q=\pi_1(V^\circ)$.

As noted by~\cite[Fact~1.2]{Nor83}, reinserting a divisor $T_{E,j}$ with its
orbifold coefficient is nothing else than imposing the relation $\gamma_{E,j}^{n_{E,j}}=1$.
At a general smooth point, the image of its loop is conjugate to
$\delta_E^{m_{E,j}}$, where $\delta_E$ is a loop around $E$.
Thus, these relations map to powers $\delta_E^{m_{E,j}n_{E,j}}$ and
their normal closure is generated by $\delta_E^{k_E}$: each exponent
is divisible by $k_E$, and B\'ezout's identity expresses $k_E$ as
an integral linear combination of these exponents.
All components of $U_{\reg}\setminus U^\circ$ reinserted at this step,
other than the horizontal boundary components already treated, dominate divisors of $V$;
any removed subsets of codimension at least two in $U_{\reg}$ do not
change the fundamental group. We obtain a commutative diagram
\[
\begin{tikzcd}[column sep=small]
 \Gam(F,R|_F)\arrow[r,"\psi_1"]\arrow[d,equal]&
 H\arrow[r,"\psi_2"]\arrow[d,"\phi_X"]&
 Q\arrow[r]\arrow[d,"\phi_C"]&1\\
 \Gam(F,R|_F)\arrow[r,"\phi_1"]&
 \Gam(U,R|_U)\arrow[r,"\phi_2"]&
 \Gam(W,R_W)\arrow[r]&1.
\end{tikzcd}
\]
The vertical maps are surjective by the definition of the orbifold
fundamental groups of pairs.

Let $\alpha\in\Gam(U,R|_U)$ such that $\phi_2(\alpha)=1$.
We take $\alpha'$ a preimage of $\alpha$ by $\phi_X$. Then
$\psi_2(\alpha')\in\ker\phi_C$, so it is a product of conjugates of
$\delta_E^{k_E}$ and their inverses.
Loops around a divisor are well-defined up to conjugation. As $\psi_2$
is surjective, we may choose conjugates $\eta_{E,j}$ of the loops
around $T_{E,j}$ such that $\psi_2(\eta_{E,j})=\delta_E^{m_{E,j}}$.
Choose integers $a_{E,j}$ satisfying
$\sum_j a_{E,j}m_{E,j}n_{E,j}=k_E$. Then
\[
 \lambda_E:=\prod_j\eta_{E,j}^{a_{E,j}n_{E,j}}\in\ker\phi_X,
 \qquad \psi_2(\lambda_E)=\delta_E^{k_E}.
\]
Therefore, taking a suitable product $\gamma'$ of conjugates of these
$\lambda_E$ and their inverses, we have
$\psi_2(\gamma')=\psi_2(\alpha')$ and $\phi_X(\gamma')=1$.
Hence, by exactness of the top row, there exists
$\beta\in\Gam(F,R|_F)$ such that $\psi_1(\beta)=\alpha'\gamma'^{-1}$.
Thus
\[
 \phi_1(\beta)=\phi_X(\psi_1(\beta))
 =\phi_X(\alpha'\gamma'^{-1})=\alpha,
\]
which proves $\operatorname{im}\phi_1\supseteq\ker\phi_2$.
The reverse inclusion follows from commutativity, and $\phi_2$ is
surjective because $\psi_2$ and $\phi_C$ are surjective.
Hence the bottom sequence is also exact.
Finally, adding the "missing divisors" in $Z_{\reg}$ possibly imposes further
loop relations, yielding $\Gam(U,R|_U)\twoheadrightarrow\Gam(Z,R)$.

\end{proof}

\begin{lemma}\label{lem:discriminant}
In the setting of Lemma~\ref{lem-fibration}, suppose $(Z,B,\bM)$ is generalized
lc Calabi-Yau and $R\le B$. Then $R_W\le B_W$, where $B_W$
is the discriminant of the canonical bundle formula. If
$(W,B_W,\bN)$ is $\mathbb Q$-factorial generalized dlt, then
$(W,R_W)$ is klt and has the same generalized lc completion.
\end{lemma}
\begin{proof}
Writing $b_{E,j}=\coeff_{T_{E,j}}B$, the generalized lc threshold satisfies
\[
 t_E\le\frac{1-b_{E,j}}{m_{E,j}}
 \le\frac1{m_{E,j}n_{E,j}}\le\frac1{k_E}.
\]
Thus $R_W\le B_W$. All coefficients of $R_W$ are smaller than one.
Lowering every coefficient-one component of a generalized dlt pair
makes the pair generalized klt since each lc place has center in the
coefficient-one boundary and its discrepancy strictly increases. Comparing the discrepancies, we get that $(W,R_W)$ is klt.
\end{proof}

\section{Fundamental Groups and Coregularity}\label{sec:coregularity}
We prove Theorem~\ref{intro:coreg-all} by reduction to standard-coefficient
Calabi-Yau pairs. 
We then apply the regular-locus theorem to a log index-one cover and
deduce the Calabi-Yau and nef cases in this section.

\begin{theorem}[Theorem~\ref{intro:coreg-all}]\label{thm:coreg-all}
Let $(X,D)$ be a projective klt pair of finite absolute
coregularity $c$. Then $\Gam(X,D)$ has a normal free abelian
subgroup of finite index and rank at most $2c$.
\end{theorem}

We recall the reduction theorem for the sake of completeness:

\begin{theorem}[Theorem~\ref{introthm-induction-coreg}]
\label{thm-induction-coreg}\label{thm:replacement-reduction}
Fix an integer $c\ge0$. Assume that for every projective klt Calabi-Yau pair
of dimension $c$ with standard coefficients the regional orbifold
fundamental group is virtually abelian of rank at most $2c$.
Then every projective klt pair $(X,D)$ of absolute coregularity
$a\le c$ has virtually abelian regional orbifold fundamental group
of rank at most $2a$.
\end{theorem}
\begin{proof}
\ \\
\noindent
\emph{Preparation.}
If the assumption holds in dimension $c$, then it also holds in every dimension $k\le c$: apply the
hypothesis to the product of the given $k$-dimensional Calabi-Yau pair
$(V,C)$ with standard coefficients and $E^{c-k}$, where $E$ is an
elliptic curve, with boundary $\operatorname{pr}_V^*C$.

If $G$ is the regional orbifold fundamental group of the given
$k$-dimensional pair, then that of the product is 
$G\times\mathbb Z^{2(c-k)}$. Virtual abelianity and the bound
$2c$ imply that $G$ is virtually abelian of rank at most $2k$.
By Lemma~\ref{lem:cy-coreg} the product has absolute coregularity $c$.

We perform an induction on $n=\dim X$ (simultaneously for all absolute
coregularities at most $c$). 
Choose a generalized lc completion
$(X,B,\bM)$ of coregularity
$a=\widehat{\operatorname{coreg}}(X,D)\le c$.
First take a small $\mathbb Q$-factorialization. The regional fundamental group
and the completion's coregularity are unchanged. Put $P=D_{\st}$,
then $(X,P)$ is klt, $P\le B$, and $\Gam(X,P)=\Gam(X,D)$.
\ \\
\noindent
\emph{Case 1: $K_X+P$ is pseudo-effective.}
Lemma~\ref{lem:cy-coreg} shows that $(X,P)$ is a klt Calabi-Yau pair
with standard coefficients and 
forces $a=n$. The assumption in dimensions $n\le c$ proves the
claim. 
\ \\
\noindent
\emph{Case 2: $K_X+P$ is not pseudo-effective.}
Run a $(K_X+P)$-MMP with scaling to a Mori fibre space
$f:Z\to W$. The completion is carried crepantly through the MMP,
since its log canonical divisor is $\mathbb Q$-linearly trivial.
It remains generalized lc of coregularity $a$. We have
$\Gam(Z,P_Z)\twoheadrightarrow\Gam(X,P)$ by
Lemma~\ref{lem-MMP-group}. If $W$ is a point, $(Z,P_Z)$ is log
Fano, so $\Gam(Z,P_Z)$ is finite by Theorem~\ref{thm:fano}.

Otherwise apply the canonical bundle formula, take a generalized
dlt modification $W'\to W$ of the base, and use
Lemma~\ref{lem-dlt-mod-Fano-fibration} to obtain $f':Z'\to W'$.
Choose the boundary $R_{Z'}$ with standard coefficients using
Lemma~\ref{lem-dlt-group}, so that
\[
 \Gam(Z',R_{Z'})\twoheadrightarrow\Gam(Z,P_Z).
\]
It satisfies $R_{Z'}\le B_{Z'}$ and has all coefficients smaller
than one. Since the completion on $Z'$ is generalized dlt, the
ordinary pair $(Z',R_{Z'})$ is klt.

Let $R_{W'}$ be the gcd boundary from Lemma~\ref{lem-fibration}.
By Lemma~\ref{lem:discriminant}, the pair $(W',R_{W'})$ is klt
and its generalized lc completion has coregularity at most $a$, by
Remark~\ref{rem:coreg-facts}. Thus its absolute coregularity is
$a'\le a$. Since $\dim W'<n$, by induction we get virtual abelianity
of rank at most $2a'\le2a$ for $\Gam(W',R_{W'})$.

Let $F'$ and $F$ be general fibres of $f'$ and $f$.
We next verify that $\Gam(F',R_{Z'}|_{F'})$ is finite,
where $R_{Z'}$ is the boundary just constructed using Lemma~\ref{lem-dlt-group}.
The pair $(F,P_Z|_F)$ is log Fano. Every prime divisor of $F$ survives on $F'$, and every additional
divisor on $F'$ is an lc place of the completion due to
Lemma~\ref{lem-dlt-mod-Fano-fibration}.
If necessary, first take a small $\mathbb Q$-factorialization
$\widehat F\to F$ so that the log anticanonical divisor becomes nef and
big, and the completion pulls back crepantly. 
By Lemma~\ref{lem-BCHM-generalized-pairs}, there is an extraction
$F_0\to\widehat F$ of exactly these places. Hence $F_0\dashrightarrow F'$
is small. The strict transform $R_{F_0}$ of $R_{Z'}|_{F'}$ on $F_0$
is the strict transform of $P_Z|_F$ plus exceptional divisors with
standard coefficients smaller than one.
Lemma~\ref{lem-Fano-extraction} gives that $\Gam(F_0,R_{F_0})$ is finite.
The small map identifies it with $\Gam(F',R_{Z'}|_{F'})$
by Lemma~\ref{lem-MMP-group}.

Apply Lemma~\ref{lem-fibration} to $f'$ and let $U=f'^{-1}(V)$
be the resulting open set. There is an exact sequence
\[
 \Gam(F',R_{Z'}|_{F'})\longrightarrow\Gam(U,R_{Z'}|_U)
 \longrightarrow\Gam(W',R_{W'})\longrightarrow1.
\]
By Lemma~\ref{lem-virtual-groups},  the middle group is virtually
abelian of rank at most $2a$. The surjections
\[
 \Gam(U,R_{Z'}|_U)\twoheadrightarrow\Gam(Z',R_{Z'})
 \twoheadrightarrow\Gam(Z,P_Z)\twoheadrightarrow\Gam(X,D)
\]
yield the desired statement.
\end{proof}

\subsection{Calabi-Yau pairs in arbitrary dimension}
The following proposition builds upon the global log-index one cover constructed in~\cite[Proposition~12]{CGG24}. 
For the sake of completeness, we provide the construction and the fundamental-group comparison explicitly.
\begin{proposition}\label{prop:cy-cover}
Let $(X,D)$ be a projective klt Calabi-Yau pair of dimension $n$
with standard coefficients. Write $D=\sum_i(1-1/m_i)D_i$, and let
$s$ be the smallest positive integer such that $s(K_X+D)\sim0$.
There is a connected finite cyclic cover $p:Y\to X$ of degree $s$,
with ramification index exactly $m_i$ over $D_i$ and no other
divisorial ramification, such that $Y$ is canonical, $K_Y\sim0$
is Cartier, and
\[
 1\longrightarrow\pi_1(Y_{\reg})\longrightarrow\Gam(X,D)
 \longrightarrow\mathbb Z/s\mathbb Z\longrightarrow1.
\]
Consequently, $\Gam(X,D)$ is virtually abelian of rank
$2\qtilde(Y)\le2n$. 
\end{proposition}
\begin{proof}
Choose a canonical Weil divisor $K_X$ and a rational function
$g\in\mathbb C(X)^*$ such that
$s(K_X+D)=\operatorname{div}_X(g)$.
Let $t$ be an $s$-th root of $g$ in an algebraic closure of
$\mathbb C(X)$, and set $L=\mathbb C(X)(t)$.
This is a cyclic Galois extension of degree $e$ dividing $s$, and
$t^e\in\mathbb C(X)$. If $e<s$, then
$\operatorname{div}_X(t^e)=e(K_X+D)$, contrary to the minimality
of $s$. Thus $[L:\mathbb C(X)]=s$.

Let $Y$ be the normalization of $X$ in $L$. 
At a prime divisor $P$ on $X$, the ramification index of this
root extension is
\[
 e_P=\frac{s}{\gcd(s,\operatorname{ord}_P(g))}.
\]
The coefficient of $K_X$ at $P$ is an integer. Also $m_i$ divides
$s$, because $s(K_X+D)$ is integral. At $D_i$, writing that
integer as $a_i$, we obtain
\[
 \operatorname{ord}_{D_i}(g)=s(a_i+1)-s/m_i,
 \qquad e_{D_i}=m_i.
\]
At a prime divisor $P$ outside $\operatorname{Supp}D$, the order
of vanishing $\operatorname{ord}_P(g)$ 
is divisible by $s$, so there is no other divisorial ramification.

Choose the canonical divisor on $Y$ using the pullback of the
rational top form defining $K_X$. At a divisor over $D_i$, the
ramification coefficient is $m_i-1=m_i(1-1/m_i)$.
The canonical divisor formula therefore gives
\[
 K_Y=p^*(K_X+D)=\frac1s\operatorname{div}_Y(g)
     =\operatorname{div}_Y(t).
\]
Only the sum $K_X+D$ is pulled back as a $\mathbb Q$-Cartier
divisor; the equality is checked in codimension one. Thus $K_Y$
is Cartier and linearly trivial. Finite crepant pullback preserves
klt singularities by~\cite[Proposition~5.20]{KM98}, and since $K_Y$ is Cartier, $Y$ is in fact canonical.

Remove a codimension-two subset of $X$ so that the remaining open
set $U$ is smooth and the branch divisors are smooth and disjoint.
Locally along $D_i\cap U$, the normalized cover is a disjoint union
of maps $(u,z_2,\ldots,z_n)\mapsto(u^{m_i},z_2,\ldots,z_n)$.
Away from the boundary it is \'etale. Thus $p^{-1}(U)$ is smooth
and the cover is an orbifold covering with the exact prescribed
ramification. The complement of $p^{-1}(U)$ in $Y$ has codimension at least two,
because $p$ is finite. Removing these codimension-two subsets from
the regular loci changes neither the fundamental group nor the
orbifold loop relations. This proves the claimed exact sequence.
Now apply~\cite[Theorem~1.3]{FGSW26} to $Y$.
\end{proof}

The global $\mathbb Q$-linear triviality and the standard coefficients
are essential to this construction. The former supplies $g$ and makes the global covering construction work, while the
latter makes the pulled-back boundary vanish. In particular,
$\mathbb Q$-Cartierness of $K_X+D$ alone does not assert the
existence of such a global cover. For arbitrary rational coefficients in the boundary,
standard approximation of course need not preserve the Calabi-Yau condition.
Theorem~\ref{thm-induction-coreg} supplies this additional needed reduction step,
and also the passage to pairs with nef log anticanonical divisor - see below.

\begin{proof}[Proof of Theorem~\ref{thm:coreg-all}] 
Proposition~\ref{prop:cy-cover} proves virtual abelianity and the rank
bound $2c$ for every projective klt Calabi-Yau pair of dimension $c$
with standard coefficients. Thus the hypothesis of
Theorem~\ref{thm-induction-coreg} holds for every $c$.
 Taking the normal core of a finite-index free abelian subgroup gives the stated form.
\end{proof}

\begin{corollary}[Corollary~\ref{intro:irregularity}]\label{cor:irregularity}
For a projective klt pair $(X,D)$ of finite absolute coregularity,
\[
 \qtilde(X)\le\widehat{\operatorname{coreg}}(X,D).
\]
\end{corollary}
\begin{proof}
For any connected finite quasi-\'etale cover $p:X'\to X$, let
$D'$ be the codimension-one pullback of $D$. Then $(X',D')$ is klt, and its regional orbifold fundamental group is a finite-index
subgroup of $\Gam(X,D)$. It therefore has the same virtual rank. 
Let $\rho:V\to X'$ be a projective resolution which is an isomorphism
over $X'_{\reg}$. Killing the boundary loops and then filling in
the exceptional divisors gives surjections
\[
 \Gam(X',D')\twoheadrightarrow\pi_1(X'_{\reg})
 \twoheadrightarrow\pi_1(V)
 \twoheadrightarrow H_1(V,\mathbb Z)/\mathrm{torsion}.
\]
The last group is isomorphic to $\mathbb Z^{2q(V)}$: Hodge decomposition and
Hodge symmetry give $b_1(V)=h^{1,0}(V)+h^{0,1}(V)=2q(V)$;
see~\cite[Corollary~3.2.12]{Huy05}.
The klt pair $(X',D')$ has rational singularities
\cite[Theorem~5.22]{KM98}, so $q(V)=q(X')$.
Hence $2q(X')\le\rank\Gam(X,D)\le2\widehat{\operatorname{coreg}}(X,D)$.
Taking the supremum of $q(X')$ over all connected finite quasi-\'etale
covers of $X$ gives $\qtilde(X)\le\widehat{\operatorname{coreg}}(X,D)$.
\end{proof}

\begin{theorem}[Theorem~\ref{intro:cy-all}]\label{cor:cy-all}
Let $(X,D)$ be a projective klt pair of dimension $n$ with rational,
possibly non-standard boundary $D$ and $K_X+D\num0$. Let $s$ be the smallest positive integer such that
$s(K_X+D)\sim0$. Then $\Gam(X,D)$ is virtually abelian of rank
at most $2n$. In particular:
\begin{enumerate}
	\item If $D$ has standard coefficients, there is a connected
	finite cyclic log index-one cover $p:Y\to X$ of degree $s$, with $Y$ canonical and $K_Y\sim0$, and an exact sequence
	\[
	1\longrightarrow\pi_1(Y_{\reg})\longrightarrow\Gam(X,D)
	\longrightarrow\mathbb Z/s\mathbb Z\longrightarrow1.
	\]
	In this case the virtual rank is exactly $2\qtilde(Y)$.
	\item If $D=0$, the virtual rank is exactly $2\qtilde(X)$.
\end{enumerate}
\end{theorem}
\begin{proof}
By~\cite[Theorem~1.2]{Gon13}, we have $K_X+D\sim_{\mathbb Q}0$, and $(X,D,0)$ is a
completion. Apply Theorem~\ref{thm:coreg-all} for the rank bound.
When $D$ has standard coefficients, Proposition~\ref{prop:cy-cover}
gives the stated cyclic cover, exact sequence, and rank $2\qtilde(Y)$.
For $D=0$, the exact rank $2\qtilde(X)$ follows directly from \cite[Theorem~1.3]{FGSW26}.
\end{proof}

\begin{example}\label{ex:rank-product}
Let $C=\mathbb P^1$, let $p_1,\ldots,p_4$ be distinct points,
and put $\Delta=\frac12\sum_i p_i$. The double cover branched
at these points is elliptic, so $\Gam(C,\Delta)$ has a subgroup
$\mathbb Z^2$ of index two. For every $n\ge1$, the klt
Calabi-Yau pair
\[
 (X,D)=\left(C^n,\sum_{i=1}^n\operatorname{pr}_i^*\Delta\right)
\]
has a subgroup $\mathbb Z^{2n}$ of index $2^n$ in its regional
orbifold fundamental group, while $\qtilde(X)=0$. We see that indeed the dimension bound is sharp. 
\end{example}

\subsection{Pairs with nef log anticanonical class}
\begin{theorem}[Theorem~\ref{intro:nef}]\label{thm:nef}
Let $(X,D)$ be a projective klt pair of dimension $n$ such that
$-(K_X+D)$ is nef. Then $\Gam(X,D)$ is virtually abelian of
rank at most $2\widehat{\operatorname{coreg}}(X,D)\le2n$.
\end{theorem}
\begin{proof}
Set $B=D$ and $\bM=\overline{-(K_X+D)}$.
The equality $K_X+D+\bM_X=0$ defines a generalized klt, hence
generalized lc, Calabi-Yau completion. The nef part descends to
$X$, so the generalized discrepancies are exactly those of $(X,D)$.
Theorem~\ref{thm:coreg-all} gives the sharp coregularity bound, and
$\widehat{\operatorname{coreg}}(X,D)\le\dim X$ gives the stated
dimension bound. 
\end{proof}

\subsection{Growth and uniformization}\label{sec:growth-uniformization}
\begin{corollary}[Corollary~\ref{intro:growth}]\label{cor:polynomial-growth}
Let $(X,D)$ be a projective klt pair with $-(K_X+D)$ nef, and
put $c=\widehat{\operatorname{coreg}}(X,D)$.
Both $\Gam(X,D)$ and $\pi_1(X_{\reg})$ have polynomial growth
of degree at most $2c$.
\end{corollary}
\begin{proof}
For a finite generating set, write $B_G(R)$ for the ball
of radius $R$ in $G$. A group with a finite-index subgroup
$\mathbb Z^r$ has word growth comparable to $(1+R)^r$:
passing to a finite-index subgroup changes word lengths by at
most multiplicative and additive constants, and balls in
$\mathbb Z^r$ have this growth. Apply Theorem~\ref{thm:nef}.
The obvious surjection
$\Gam(X,D)\twoheadrightarrow\pi_1(X_{\reg})$ gives that the latter
group is also virtually abelian of rank at most $2c$.
\end{proof}

The following observation identifies the rank of the ordinary fundamental group of the smooth locus. It does not identify the orbifold rank; compare Example~\ref{ex:rank-product}.
\begin{lemma}\label{lem:regular-rank}
Let $(X,D)$ be a projective klt pair. If $\pi_1(X_{\reg})$ is
virtually abelian of rank $r$, then $r=2\qtilde(X)$.
There is a connected finite quasi-\'etale cover $X_0\to X$ with
\[
 \pi_1((X_0)_{\reg})\xrightarrow{\ \simeq\ }\pi_1(X_0)
 \simeq\mathbb Z^r.
\]
Every connected finite quasi-\'etale cover of $X_0$ is \'etale.
\end{lemma}
\begin{proof}
Choose a finite-index subgroup $\mathbb Z^r$ of $\pi_1(X_{\reg})$.
Its connected finite topological cover is algebraic. Normalize $X$ in its function field to
obtain a finite quasi-\'etale cover $X_0\to X$.

The preimage of $X_{\reg}$ is a smooth open subset of $X_0$
whose complement has codimension at least two. Thus
$\pi_1((X_0)_{\reg})=\mathbb Z^r$.
The pulled-back pair is moreover klt. By the local finiteness
theorem~\cite{Bra21}, its regional fundamental groups
are finite, hence have trivial image in $\mathbb Z^r$.
The kernel of the surjection
$\pi_1((X_0)_{\reg})\to\pi_1(X_0)$ is normally generated by
these local images, so this map is an isomorphism.

For a projective log resolution $V\to X_0$ that is an
isomorphism over $(X_0)_{\reg}$, the kernel of
$\pi_1((X_0)_{\reg})\to\pi_1(V)$ is normally generated by
loops around the exceptional divisors. Each loop lies
in the image of a regional fundamental group and is therefore
trivial. Consequently $\pi_1(V)=\mathbb Z^r$.
Rational singularities and Hodge theory now give
$r=b_1(V)=2q(V)=2q(X_0)$, as in the proof of
Corollary~\ref{cor:irregularity}.

For any connected finite quasi-\'etale cover $W\to X$,
$\pi_1(W_{\reg})$ is a finite-index subgroup of
$\pi_1(X_{\reg})$. Its surjection onto the first homology
modulo torsion of a resolution of $W$ gives $2q(W)\le r$.
Together with the equality for $X_0$, this proves
$r=2\qtilde(X)$.
Finally, the isomorphism
$\pi_1((X_0)_{\reg})\simeq\pi_1(X_0)$ implies that every
finite \'etale cover of $(X_0)_{\reg}$ extends to a finite \'etale
cover of $X_0$. The uniqueness of normalization finally shows that a
finite quasi-\'etale cover of $X_0$ is this \'etale extension.
\end{proof}

\begin{theorem}[Theorem~\ref{intro:nef-uniformization}]
\label{thm:nef-uniformization}
Let $(X,D)$ be a projective klt pair with $-(K_X+D)$ nef, and
put $q=\qtilde(X)$. Then $\pi_1(X_{\reg})$ is virtually
$\mathbb Z^{2q}$, and it is finite if and only if $q=0$.
There is a connected finite quasi-\'etale cover $p:X'\to X$
such that
\[
 \pi_1(X'_{\reg})\xrightarrow{\ \simeq\ }\pi_1(X')
 \simeq\mathbb Z^{2q},\qquad
 (\widetilde X',\widetilde D')\simeq
 \mathbb C^q\times Z\times(F,D_F).
\]
Here $\widetilde X'\to X'$ is the ordinary universal cover,
$\widetilde D'$ is the pullback of $D'=p^*D$, $Z$ is
projective and canonical with $K_Z\sim0$, and $(F,D_F)$ is
projective klt with $F$ rationally connected and
$-(K_F+D_F)$ nef. Moreover,
$\pi_1(Z_{\reg})=\pi_1(F_{\reg})=1$.
Either factor may be a point.
\end{theorem}
\begin{proof}
Virtual abelianity is obvious from theorem~\ref{thm:nef} and the surjection to
$\pi_1(X_{\reg})$. Apply Lemma~\ref{lem:regular-rank} to obtain $r=2q$ and $X_0\to X$.
The asserted finiteness criterion is clear as well.

Apply~\cite[Theorem~1.1]{MW25} to the pulled-back nef pair
on $X_0$. After a connected finite quasi-\'etale cover, its
MRC fibration is a locally constant fibration of pairs over
a projective klt variety $Y$ with $K_Y\num0$.
By~\cite[Theorem~1.3]{FGSW26}, there is a connected finite
quasi-\'etale cover $A\times Z\to Y$, where $A$ is abelian,+$Z$ is projective canonical, $K_Z\sim0$, and
$\pi_1(Z_{\reg})=1$.
Base change gives a connected finite quasi-\'etale cover
$X'\to X_0$ and a locally constant fibration of pairs
\[
 (X',D')\longrightarrow A\times Z
\]
with fiber $(F,D_F)$. These properties of the base change
follow from the local product description: the fibers are
connected and normal, and inverse images of the exceptional
codimension-two subset of the base still have codimension
at least two.
The pair $(F,D_F)$ is klt, $F$ is rationally connected,
and $-(K_F+D_F)$ is nef by restriction.

By Lemma~\ref{lem:regular-rank}, the composite cover
$X'\to X_0$ is \'etale. In particular,
$\pi_1(X'_{\reg})\to\pi_1(X')$ is an isomorphism and both
groups are $\mathbb Z^{2q}$.
A rationally connected projective variety of klt type is
simply connected; see~\cite[Theorem~2.2]{MW25}.
Thus $\pi_1(F)=1$. Also $\pi_1(Z)=1$, since
$\pi_1(Z_{\reg})\twoheadrightarrow\pi_1(Z)$.
The homotopy sequence of the locally trivial fibration
identifies $\pi_1(X')$ with $\pi_1(A\times Z)$, so
$\dim A=q$. Pulling back to the universal cover
$\mathbb C^q\times Z$ of the base trivializes the pair
fibration, and its simply connected total space is
\[
 (\widetilde X',\widetilde D')\simeq
 \mathbb C^q\times Z\times(F,D_F).
\]
Finally, the restriction of this universal cover to
$X'_{\reg}$ is connected and simply connected, because
$\pi_1(X'_{\reg})\to\pi_1(X')$ is an isomorphism.
Its total space is $\mathbb C^q\times Z_{\reg}\times F_{\reg}$.
Taking fundamental groups of this product gives
$\pi_1(F_{\reg})=1$.
\end{proof}

In particular, the polynomial-growth degree of
$\pi_1(X_{\reg})$ is exactly $2\qtilde(X)$, whereas the
degree for $\Gam(X,D)$ can be larger.
By the same arguments and Theorem~\ref{thm:coreg-all}, the rank
formula and the finiteness criterion for $\pi_1(X_{\reg})$
also hold for every projective klt pair of finite absolute
coregularity, without assuming nefness.

\section{Effective Bounds}\label{sec:eff-virt-solv}
In this section, we prove the effective result for small absolute
coregularities. We note that no upper bound on the order of the finite group
$\Gam(F',R_{Z'}|_{F'})$ of a general fibre in the proof of
Theorem~\ref{thm-induction-coreg} is required.
Instead, Lemma~\ref{lem-Fano-extraction} supplies a uniform
bound on the index of an abelian subgroup of that group.

\begin{proposition}\label{prop-induction-solvable}
Fix an integer $c_0\ge0$. Suppose that for every integer $1\le k\le c_0$
there is a positive integer $I_k$ such that the regional orbifold
fundamental group $\Gam(V,\Delta)$
of every projective klt Calabi-Yau pair $(V,\Delta)$ of dimension $k$
with standard coefficients
has a solvable subgroup of derived length at most $2k-1$ and index
at most $I_k$, independently of the pair.
Then for every integer $d\ge1$ there is a positive integer
$B(d,c_0)$, depending only on $d,c_0$ and the constants $I_1,\ldots,I_{c_0}$,
such that for every projective klt pair $(X,D)$ of dimension $d$
and absolute coregularity at most $c_0$, the group $\Gam(X,D)$ has
a solvable subgroup of derived length at most $2d-1$ and index
at most $B(d,c_0)$.
\end{proposition}
\begin{proof}
Repeat the induction and constructions from the proof of
Theorem~\ref{thm-induction-coreg}, writing $P=D_{\st}$ on a small
$\mathbb Q$-factorialization of $X$.
If $K_X+P$ is pseudo-effective, Lemma~\ref{lem:cy-coreg} shows
that $(X,P)$ is Calabi-Yau and $d$ equals its absolute coregularity,
so $d\le c_0$.
The hypothesis of this proposition, with $k=d$, gives the required
solvable subgroup of $\Gam(X,P)=\Gam(X,D)$ with index at most $I_d$.
If the $(K_X+P)$-MMP ends in a Mori fibre space $Z\to W$ with
$W$ a point, $(Z,P_Z)$ is log Fano. Theorem~\ref{thm:fano} shows
that $\Gam(Z,P_Z)$ is finite and has an abelian subgroup of index
at most $J(d)$. Its image in $\Gam(X,D)$ has the same index bound.
For a base of dimension $1\le w<d$, the proof of
Theorem~\ref{thm-induction-coreg} constructs the exact sequence
\[
 \Gam(F',R_{Z'}|_{F'})\longrightarrow\Gam(U,R_{Z'}|_U)
 \longrightarrow\Gam(W',R_{W'})\longrightarrow1,
\]
where $U\subset Z'$ is the indicated open set, $F'$ is a general
fibre of $Z'\to W'$, and $\Gam(U,R_{Z'}|_U)$ surjects onto $\Gam(X,D)$.
The group $\Gam(F',R_{Z'}|_{F'})$ is finite and has an abelian
subgroup of index at most $J(d-w)$, by Lemma~\ref{lem-Fano-extraction}
and the comparison with $F_0$ in the proof of Theorem~\ref{thm-induction-coreg}.
By induction, $\Gam(W',R_{W'})$ has a solvable subgroup of derived
length at most $2w-1$ and index at most $B(w,c_0)$.
Lemma~\ref{lem-virt-sol} therefore gives a solvable subgroup of
$\Gam(U,R_{Z'}|_U)$ of derived length at most $2w+1\le2d-1$
and index at most $B(w,c_0)(J(d-w)^2)!$.
For any surjection $\rho:G\to Q$ and subgroup $S\le G$, the
image $\rho(S)$ has derived length at most that of $S$ and
$[Q:\rho(S)]\le[G:S]$. Thus the same bounds hold for the image
subgroup in $\Gam(X,D)$.
These estimates define $B(d,c_0)$ recursively. For $d=1$,
Proposition~\ref{prop-dim1} gives an abelian subgroup of index at
most $60$ for every curve pair under consideration. A connected
zero-dimensional normal variety is a point and has trivial fundamental group.
\end{proof}

\begin{proof}[Proof of Theorem~\ref{introthm-solv012}]
For projective klt Calabi-Yau pairs with standard coefficients,
Propositions~\ref{prop-dim1} and \ref{prop-dim2} give abelian
subgroups of their regional orbifold fundamental groups of indices
at most $60$ in dimension one and $3840$ in dimension two.
These verify the hypotheses of Proposition~\ref{prop-induction-solvable}
in all dimensions $1\le k\le2$; no higher-dimensional Calabi-Yau
index bound is needed when $c_0=2$.
Apply Proposition~\ref{prop-induction-solvable} with $c_0=2$.
For an explicit recursive choice of bounds, enlarge $J$ to be
integer-valued and nondecreasing, put $b_0=1$, and set
\[
 b_d=\max\left\{3840,\,J(d),\,
       \max_{1\le w<d} b_w\bigl(J(d-w)^2\bigr)!\right\}
 \quad(d\ge1),
\]
where the inner maximum is omitted for $d=1$.
The induction gives index at most $b_d$ for a solvable subgroup.
Its normal core has the same derived-length bound and index at
most $b_d!$. Thus one may take $i(c,d)=b_d!$ for $c\le2$.
\end{proof}

\section{Torus Covers}\label{sec:new-results}
\begin{theorem}[Theorem~\ref{introthm-torus-covers}]\label{thm:torus}
    Let $(X,D)$ be a projective klt pair with $K_X+D \sim_{\mathbb{Q}} 0$. Then there exists a connected finite quasi-\'etale cover $\nu:Y \to X$, such that $Y= A \times Z$ has a product structure with the following properties:
    \begin{enumerate}
        \item $A$ is an Abelian variety of dimension $\tilde{q}(X)$.
        \item There is a boundary $D_Z$ on $Z$ such that $(Z,D_Z)$ is klt and $K_Z+D_Z\sim_{\mathbb{Q}} 0$. Moreover, $\tilde{q}(Z)=0$ and $\nu^*D=\operatorname{pr}_Z^*D_Z$.
    \end{enumerate}
\end{theorem}
\begin{proof}
Every connected finite quasi-\'etale cover $X'\to X$, with the
pulled-back boundary, is again klt Calabi-Yau. By
\cite[Theorem~4.8]{Amb05}, its Albanese map is surjective, so
$q(X')\le\dim X$. The supremum defining $\qtilde(X)$ is
therefore a maximum. Choose $X'\to X$ with $q(X')=\qtilde(X)$.
Ambro's theorem also provides a finite \'etale base change
$A\to\operatorname{Alb}(X')$ after which the pair splits as
\[
 (Y,D_Y)\simeq(A\times Z,\operatorname{pr}_Z^*D_Z).
\]
The composite $Y\to X$ is finite quasi-\'etale, and
$\dim A=q(X')=\qtilde(X)$. The pair $(Z,D_Z)$, obtained by
restricting $(Y,D_Y)$ to a fibre of the projection $Y\to A$,
is klt Calabi-Yau. 
If $Z$ had a finite quasi-\'etale cover $Z'$ with $q(Z')>0$,
then $A\times Z'\to X$ would be a finite quasi-\'etale cover
with irregularity larger than $\qtilde(X)$, a contradiction.
Thus $\qtilde(Z)=0$.
\end{proof}

\begin{remark}\label{rem:terminalization}
For boundary-free klt Calabi-Yau varieties, the criterion
$\qtilde(X)=0$ is equivalent to finiteness of $\pi_1(X_{\reg})$
in every dimension, by~\cite{FGSW26}. 
For a terminal threefold with $K_X\sim0$ Cartier, one also has
the following alternative:
$X$ is smooth, or is a finite quasi-\'etale quotient of an abelian
threefold, or has finite regional fundamental group. Indeed,
$\qtilde(X)=0$ gives the last case, and $\qtilde(X)=3$ gives the
abelian cover. In the remaining case, the torus cover has a terminal
factor of dimension at most two, hence is smooth. Thus $X$ has
quotient singularities. Since $K_X$ is Cartier, a local quotient
presentation may be chosen with a finite group $G\subset\mathrm{SL}_3(\mathbb C)$
and no nontrivial quasi-reflections.
For an element $g$ of order $r$ with eigenvalues
$\exp(2\pi i a_j/r)$, $0\le a_j<r$, define
$\operatorname{age}(g)=\sum_{j=1}^3a_j/r$.
The Reid-Tai criterion says that a quotient without quasi-reflections
is terminal if and only if $\operatorname{age}(g)>1$ for every
nonidentity element $g$; see~\cite[Theorem~4.11]{Rei87}.
For $g\in\mathrm{SL}_3$, both $\operatorname{age}(g)$ and
$\operatorname{age}(g^{-1})$ are integers, and
\[
 \operatorname{age}(g)+\operatorname{age}(g^{-1})
 =\#\{j:a_j\ne0\}\le3.
\]
Terminality would make both integers at least two, so their sum would
be at least four, contradicting the displayed bound.
Consequently $G$ is trivial, and a Gorenstein terminal quotient threefold is smooth.
\end{remark}


\begin{thebibliography}{BCHM10}
\raggedright
\bibitem[Amb05]{Amb05} F.~Ambro,
\emph{The moduli b-divisor of an lc-trivial fibration},
Compos. Math. \textbf{141} (2005), 385-403.
\href{https://arxiv.org/abs/math/0308143}{arXiv:math/0308143}.

\bibitem[BCHM10]{BCHM10} C.~Birkar, P.~Cascini, C.~D.~Hacon, and J.~McKernan,
\emph{Existence of minimal models for varieties of log general type},
J. Amer. Math. Soc. \textbf{23} (2010), 405-468.

\bibitem[BF24]{BF24} L.~Braun and F.~Figueroa,
\emph{Fundamental Groups, Coregularity, and Low Dimensional Klt
Calabi Yau Pairs}, preprint, version 2, 21 May 2024.
\href{https://arxiv.org/abs/2401.01315v2}{arXiv:2401.01315v2}.

\bibitem[BFMS22]{BFMS20} L.~Braun, S.~Filipazzi, J.~Moraga, and R.~Svaldi,
\emph{The Jordan property for local fundamental groups},
Geom. Topol. \textbf{26} (2022), 283-319.
\href{https://arxiv.org/abs/2006.01253v2}{arXiv:2006.01253v2}.

\bibitem[Bra21]{Bra21} L.~Braun,
\emph{The local fundamental group of a Kawamata log terminal singularity
is finite}, Invent. Math. \textbf{226} (2021), 845-896.
\href{https://arxiv.org/abs/2004.00522}{arXiv:2004.00522}.




\bibitem[CD89]{CD89} A.~Chermak and A.~Delgado,
\emph{A measuring argument for finite groups},
Proc. Amer. Math. Soc. \textbf{107} (1989), 907-914.

\bibitem[CGG24]{CGG24} B.~Claudon, P.~Graf, and H.~Guenancia,
\emph{Equality in the Miyaoka-Yau inequality and uniformization of
non-positively curved klt pairs}, C. R. Math. Acad. Sci. Paris
\textbf{362} (2024), no.~S1, 55-81.
\href{https://doi.org/10.5802/crmath.599}{doi:10.5802/crmath.599}.

\bibitem[FGSW26a]{FGSW26nef} X.~Fu, B.~Guo, J.~Song, and J.~Wang,
\emph{Fundamental groups of compact K\"ahler varieties with nef
anticanonical bundle}, preprint, 7 February 2026.
\href{https://arxiv.org/abs/2602.07420v1}{arXiv:2602.07420v1}.

\bibitem[FGSW26b]{FGSW26} X.~Fu, B.~Guo, J.~Song, and J.~Wang,
\emph{Fundamental groups of Calabi-Yau spaces}, preprint, 7 October 2026.
\href{https://arxiv.org/abs/2610.09269v1}{arXiv:2610.09269v1}.

\bibitem[FMP25]{FMP22} F.~Figueroa, J.~Moraga, and J.~Peng,
\emph{Log canonical thresholds and coregularity}, Doc. Math.
\textbf{30} (2025), 183-218.
\href{https://arxiv.org/abs/2204.05408}{arXiv:2204.05408}.

\bibitem[FS23]{FS23} S.~Filipazzi and R.~Svaldi,
\emph{On the connectedness principle and dual complexes for generalized pairs},
Forum Math. Sigma \textbf{11} (2023), Paper No.~e33.
\href{https://arxiv.org/abs/2010.08018}{arXiv:2010.08018}.

\bibitem[Fil20]{Fil20} S.~Filipazzi,
\emph{On a generalized canonical bundle formula and generalized adjunction},
Ann. Sc. Norm. Super. Pisa Cl. Sci. (5) \textbf{21} (2020), 1187-1221.
\href{https://arxiv.org/abs/1807.04847}{arXiv:1807.04847}.

\bibitem[GLM25]{GLM23} C.~Gachet, Z.~Liu, and J.~Moraga,
\emph{Fundamental groups of log Calabi-Yau surfaces}, preprint,
revised 11 February 2025.
\href{https://arxiv.org/abs/2312.03981v2}{arXiv:2312.03981v2}.

\bibitem[Gon13]{Gon13} Y.~Gongyo,
\emph{Abundance theorem for numerically trivial log canonical divisors
of semi-log canonical pairs}, J. Algebraic Geom. \textbf{22} (2013),
549-564. \href{https://arxiv.org/abs/1005.2796}{arXiv:1005.2796}.

\bibitem[GSS26]{GSS26} B.~Guo, J.~Song, and J.~Sturm,
\emph{Ricci entropy, RCD structures and K\"ahler spaces},
preprint, 24 September 2026.
\href{https://arxiv.org/abs/2609.28986v1}{arXiv:2609.28986v1}.

\bibitem[HL21]{HL21} J.~Han and W.~Liu,
\emph{On a generalized canonical bundle formula for generically finite morphisms},
Ann. Inst. Fourier \textbf{71} (2021), 2047-2077.
\href{https://doi.org/10.5802/aif.3437}{doi:10.5802/aif.3437}.

\bibitem[Huy05]{Huy05} D.~Huybrechts,
\emph{Complex Geometry: An Introduction}, Universitext,
Springer-Verlag, Berlin, 2005.
\href{https://doi.org/10.1007/b137952}{doi:10.1007/b137952}.


\bibitem[KM98]{KM98} J.~Koll\'ar and S.~Mori,
\emph{Birational Geometry of Algebraic Varieties}, Cambridge Tracts
in Mathematics, vol.~134, Cambridge University Press, 1998.

\bibitem[Mor21]{Mor21a} J.~Moraga,
\emph{On a toroidalization for klt singularities}, preprint (2021).
\href{https://arxiv.org/abs/2106.15019v3}{arXiv:2106.15019v3}.

\bibitem[Mor24]{Mor22} J.~Moraga,
\emph{Coregularity of Fano varieties}, Geom. Dedicata
\textbf{218} (2024), Paper No.~40.
\href{https://arxiv.org/abs/2206.10834}{arXiv:2206.10834}.

\bibitem[MW25]{MW25} S.~Matsumura and J.~Wang,
\emph{Structure theorem for projective klt pairs with nef
anti-canonical divisor}, J. Eur. Math. Soc. (2025),
published online first.
\href{https://doi.org/10.4171/JEMS/1702}{doi:10.4171/JEMS/1702}.

\bibitem[MWWZ26]{MWWZ26} S.~Matsumura, J.~Wang, X.~Wu, and Q.~Zhang,
\emph{Beauville--Bogomolov--Yau decomposition for K\"ahler
generalized pairs}, preprint, revised 9 September 2026.
\href{https://arxiv.org/abs/2506.23218v2}{arXiv:2506.23218v2}.

\bibitem[Nor83]{Nor83} M.~V.~Nori,
\emph{Zariski's conjecture and related problems}, Ann. Sci. \'Ecole Norm.
Sup. (4) \textbf{16} (1983), 305-344.

\bibitem[Pau97]{Pau97} M.~P\u{a}un,
\emph{Sur le groupe fondamental des vari\'et\'es k\"ahl\'eriennes compactes
\`a classe de Ricci num\'eriquement effective},
C. R. Acad. Sci. Paris S\'er. I Math. \textbf{324} (1997), 1249-1254.

\bibitem[Rei87]{Rei87} M.~Reid,
\emph{Young person's guide to canonical singularities},
in Algebraic Geometry, Bowdoin, 1985, Proc. Sympos. Pure Math.
\textbf{46}, Part~1, Amer. Math. Soc., 1987, 345-414.

\bibitem[Wan22]{Wan22} J.~Wang,
\emph{Structure of projective varieties with nef anticanonical
divisor: the case of log terminal singularities},
Math. Ann. \textbf{384} (2022), 47-100.
\href{https://doi.org/10.1007/s00208-021-02275-7}{doi:10.1007/s00208-021-02275-7}.

\end{thebibliography}
\end{document}